\documentclass{amsart}

\usepackage{anysize}
\marginsize{1in}{1in}{1in}{1in}

\usepackage{amsfonts,amssymb,latexsym,amsmath,graphicx}

\usepackage[bookmarks=true, bookmarksopen=true,%
bookmarksdepth=3,bookmarksopenlevel=2,%
colorlinks=true,%
linkcolor=blue,%
citecolor=blue,%
filecolor=blue,%
menucolor=blue,%
urlcolor=blue]{hyperref}

\usepackage[all, knot, poly]{xy}
\xyoption{knot}

\usepackage{times}
\usepackage{mathrsfs}

\newtheorem{lemma}{Lemma}

\newtheorem{corollary}[lemma]{Corollary}

\newtheorem*{Theorem}{Theorem}

\theoremstyle{definition}

\theoremstyle{remark} 
\newtheorem*{remark}{Remark}

\newcommand{\C}{\mathbb{C}}

\newcommand{\G}{\mathbb{G}}

\newcommand{\Q}{\mathbb{Q}}

\newcommand{\cA}{\mathcal{A}}

\newcommand{\cC}{\mathcal{C}}

\newcommand{\cE}{\mathcal{E}}
\newcommand{\cF}{\mathcal{F}}

\newcommand{\cH}{\mathcal{H}}

\newcommand{\cM}{\mathcal{M}}

\newcommand{\cT}{\mathcal{T}}
\newcommand{\cU}{\mathcal{U}}

\DeclareMathOperator{\Hom}{Hom}

\author{Vivek Shende}

\title{The weights of the tautological classes of character varieties}

\makeatletter
\def\blfootnote{\gdef\@thefnmark{}\@footnotetext}
\makeatother

\begin{document}

\begin{abstract} I calculate the weights of the tautological classes of character varieties using the  
functorial mixed Hodge structure on simplicial schemes. \end{abstract}

\maketitle

\blfootnote{UC Berkeley, Dept. of Mathematics; 970 Evans Hall, Berkeley CA 94720-3840 USA.}
\blfootnote{{\tt vivek@math.berkeley.edu}}

Let $X$ be a topological space of finite type and let $G$ be a linear algebraic group over $\C$.  We write $$Loc_G(X) = \Hom(\Pi_1(X), G)/G$$ 
for the stack of locally constant principal $G$-bundles on $X$.  It is algebraic of finite type since $\Pi_1(X)$ is finitely generated.  There is a tautological $G$-bundle
on $ Loc_G(X) \times X$, given by taking the tautological flat bundle and forgetting the flat structure.  This induces a classifying map 
$$Loc_G(X) \times X \to BG$$

Passing to cohomology, we get a map $\mathrm{H}^*(BG, \Q) \to \mathrm{H}^*(Loc_G(X), \Q) \otimes \mathrm{H}^*(X, \Q)$.  Taking the Kunneth
components of the image classes on the first factor give the {\em tautological classes} of $ \mathrm{H}^*(Loc_G(X), \Q)$.  Given
$C \in \mathrm{H}^*(X, \Q)^\vee$ and $\xi \in \mathrm{H}^*(BG, \Q)$, we write $\int_C \xi$ for the corresponding tautological class; of course
in the case of usual interest when $X$ is a manifold, $C$ can be taken to be a cycle and the integral is an integral. 

$Loc_G(X)$ is an Artin stack over $\C$, which can be viewed as a simplicial scheme over $\C$.\footnote{For a detailed explanation of how to interpret Artin stacks as simplicial schemes -- 
ie., how
to define sheaves, and in particular Betti cohomology, on the former via the latter -- see \cite{O}; for a more general treatment in the context of higher stacks see \cite{P}.    
For the reader, we recall in the appendix a little bit of the basic language of simplicial schemes and explain in particular how to view $Loc$ as a simplicial scheme.} 
The Betti cohomology of a complex simplicial scheme carries
a mixed Hodge structure \cite[Def. 8.3.4]{D3}.  
The data of this structure is two filtrations: an increasing weight filtration $W_\bullet$ on the
cohomology with coefficients in $\Q$, 
and a decreasing Hodge filtration $F^\bullet$ on the  cohomology with coefficients in $\C$.
We are particularly interested in the Hodge classes
$$ {}^k\mathrm{Hdg}^{d}(Z) := F^k \mathrm{H}^d(Z, \C) \cap \overline{F}^k \mathrm{H}^d(Z, \C) \cap W_{2k} \mathrm{H}^d(Z, \Q)$$
The key property of Hodge structures is that any maps on cohomology induced by algebraic maps preserve both  filtrations \cite[Prop 8.3.9]{D3}, hence in particular $ {}^k\mathrm{Hdg}^{*}$.  

Deligne showed  $\mathrm{H}^*(BG, \Q) = \bigoplus {}^{k} \mathrm{Hdg}^{2k}(BG)$; in particular,
this Hodge structure is pure \cite[Thm. 9.1.1]{D3}.
However, the classifying map above
{\em is not algebraic} in the second factor, and so need not preserve Hodge structures.  In fact, it does not: for $G = \mathrm{PGL}_r$, and 
$\Sigma$ a closed orientable 2-manifold, and $\xi = c_k(\mathcal{T})$ the Chern class of the tautological bundle, 
Hausel and Rodriguez-Villegas showed \cite[Prop. 4.1.8]{HRV}:

\addtocounter{equation}{-1}

\begin{equation} \label{eq:pt} \int\limits_{point} c_k(\mathcal{T}) \in {}^{k} \mathrm{Hdg}^{2k}(Loc_{\mathrm{PGL_r}}(\Sigma))  \end{equation}
\begin{equation} \label{eq:curve} \int\limits_{curve} c_k(\mathcal{T}) \in {}^{k} \mathrm{Hdg}^{2k-1}(Loc_{\mathrm{PGL_r}}(\Sigma)) \end{equation}

They moreover showed for $k=2$ and conjectured in general
 that \cite[Rem. 4.1.9]{HRV}:

\begin{equation} \label{eq:surface} \int_{\Sigma} c_k(\mathcal{T}) \in {}^{k} \mathrm{Hdg}^{2k-2}(Loc_{\mathrm{PGL_r}}(\Sigma)) \end{equation}

The purpose of this note is to prove (\ref{eq:surface}), and in fact the analogous result for any $X, G, \xi, C$:

\begin{Theorem}
For $C \in \mathrm{H}^i(X, \Q)^\vee$ and $\xi \in \mathrm{H}^{2k}(BG, \Q)$, we have 
$$\int_C \xi \in {}^{k} \mathrm{Hdg}^{2k-i} (Loc_G(X))$$ 
\end{Theorem}
\begin{proof}
Let $\Delta_X$ be a simplicial set with geometric realization homotopic to $X$.  View it as a constant simplicial scheme. 

There is an algebraic morphism $\underline{\Hom}_{sschemes}(\Delta_X, BG) \xrightarrow{\sim} Loc_G(X)$, where 
the LHS is the internal hom in the category of simplicial schemes.  It is a homotopy equivalence. 
We recall the description of this map in the appendix; a proof that it is an 
equivalence can be found here: \cite[Lem. 2.2.6.3]{HAG2}.  In that reference they 
prove a stronger result accounting in addition for the derived structure, 
but it specializes to the present statement. 

There is an evaluation map of simplicial schemes
$$\underline{\Hom}_{sschemes}(\Delta_X, BG) \times \Delta_X \to BG$$ 
Its algebraicity is guaranteed by the existence of internal hom in the category of simplicial schemes,
since it corresponds to the identity under the identification 
$$  \Hom_{ssch.}(\underline{\Hom}_{ssch.}(\Delta_X, BG), 
\underline{\Hom}_{ssch.}(\Delta_X, BG)) = 
\Hom_{ssch.}(\underline{\Hom}_{ssch.}(\Delta_X, BG) \times \Delta_X, BG)$$

Simplicial schemes carry mixed Hodge structures, functorial with respect to algebraic maps 
\cite[Def. 8.3.4, Prop 8.3.9]{D3}. Thus the
induced map
$$\mathrm{H}^*(BG, \Q) \to \mathrm{H}^*(\underline{\Hom}_{sschemes}(\Delta_X, BG ), \Q) \otimes \mathrm{H}^*(\Delta_X, \Q)$$
respects the mixed Hodge structure.  $\mathrm{H}^*(\Delta_X, \Q)$ is $\mathrm{H}^*(X,\Q)$ as a graded vector space, but 
carries the mixed Hodge structure in which everything has weight zero.   Thus $\int_C$ decreases the cohomological degree
by $\deg C$ but does not change the Hodge degrees. 
\end{proof}

%

\begin{remark}
A similar statement about Hodge degrees holds with $BG$ replaced by any simplicial scheme.  However the fact that the 
mapping simplicial scheme $\underline{\Hom}_{sschemes}(\Delta_X, BG )$ is a 1-stack is special to $BG$. 
\end{remark}

Our interest in these classes is largely driven by various conjectures on the cohomology 
of $GL_n$ character varieties of surfaces which were formulated in \cite{HRV} and \cite{dCHM}. 
As the actual character variety is singular or stacky depending on one's viewpoint, the cleanest
conjectures have been made in the related setting of variants called twisted character varieties. 
In \cite[Sec. 2.2]{HRV} it is explained how to use the honest $PGL_n$ character variety to understand
the twisted versions, but for completeness, we recall the relevant definitions and manipulations. 

If $\Sigma$ is a topological surface of genus $g$ with one puncture, then its 
fundamental group is free on $2g$ generators.  The $GL_n$ representations of this are then just 
$2g$-tuples of invertible matrices.  Taking the monodromy around the puncture gives a 
map $(GL_n)^{2g} \to SL_n$; we denote $\cU_n$ 
the preimage of the scalar matrix whose entries are some primitive $n$-th root of unity.  The 
conjugation action of $GL_n$ on $\cU_n$ 
factors through a free action of $PGL_n$, and the resulting quotient space is 
smooth but noncompact, and is called the {\em twisted $GL_n$ character variety} $\cM_n$.  The above facts hold verbatim replacing $GL_n$ by $SL_n$; giving
the twisted $SL_n$ character variety $\cM'_n$.   Finally we can divide
out by the scaling action on each matrix: 
$\widetilde{\cM}_n := \cM_n / (\G_m)^{2g} = \cM'_n /(\mu_n)^{2g}$.  This last object, the
twisted $PGL_n$ character variety, is in fact contained in what we have
been calling $Loc_{PGL_n}(\Sigma)$: the scalar matrix we were demanding 
to be the monodromy around the puncture becomes the identity in $PGL_n$, so we are to 
begin with parameterizing $PGL_n$ representations of $\pi_1(\Sigma)$, and then we have
divided out by $PGL_n$ conjugation.  One can show it is a connected component. 
To be completely precise, in \cite{HRV}, $\widetilde{\cM}_n$
was viewed as a variety with orbifold singularities, and the corresponding component in 
$Loc_{PGL_n}(\Sigma)$ is a smooth Deligne-Mumford stack, but they have
the same rational cohomology. 

Thus the tautological classes we have been studying on $Loc_{PGL_n}(\Sigma)$ pull
back via 
$$\cM'_n \hookrightarrow \cM_n \to \widetilde{\cM}_n \hookrightarrow 
Loc_{PGL_n}(\Sigma)$$
to the classes used in \cite{HRV, dCHM} to study these twisted character varieties. 
We recall that, because $\cM_n = (\cM'_n \times (\G_m)^{2g})/\mu_n^{2g}$, and because
$\cM'_n/ \mu_n^{2g} = \widetilde{\cM}_n$ and $\mu_n^{2g}$ acts cohomologically trivially
on $(\G_m)^{2g}$, we have 
$\mathrm{H}^*(\cM_n; \Q) \cong 
\mathrm{H}^*(\widetilde{\cM}_n; \Q) \otimes \mathrm{H}^*((\G_m)^{2g}; \Q)$. 

The importance of the tautological classes comes from a result of Markman: 
the tautological classes, together with the
$\mathrm{H}^*((\G_m)^{2g}; \Q)$, generate the cohomology of $\cM_n$; or equivalently,
the tautological classes generate the cohomology of $\widetilde{\cM}_n$
\cite[Sec. 4]{M}.

More precisely, let us choose variables $\alpha_k, \phi_{k, i}, \beta_k$, where
$k = 2, \ldots, n$ and $i = 1, \ldots, 2g$.  
Choosing a basis $\{\gamma_i\}$ of $\mathrm{H}^1(\Sigma, \Q)$
determines a ring homomorphism
\begin{eqnarray*} 
\Q[\alpha_k, \phi_{k,i}, \beta_k] & \to & \mathrm{H}^*(\widetilde{\cM}_n, \Q) \\
\alpha_k & \mapsto & \int_\Sigma c_k(\cT) \in {}^k \mathrm{Hdg}^{2k-2}(\widetilde{\cM}_n)  \\
\phi_{k, i} & \mapsto & \int_{\gamma_i} c_k(\cT) 
 \in {}^k \mathrm{Hdg}^{2k-1}(\widetilde{\cM}_n)\\
\beta_k & \mapsto & \int_{pt} c_k(\cT)  \in {}^k \mathrm{Hdg}^{2k}(\widetilde{\cM}_n)
\end{eqnarray*}
Markman's theorem implies that this morphism is surjective.  The next natural task
is to study the relations.  These are known explicitly for $n = 2$ \cite{HT}, but in general 
remain mysterious.  

\begin{corollary}
The inclusion
$\bigoplus {}^k \mathrm{Hdg}^d(\widetilde{\cM}_n) \subset \mathrm{H}^{*}(\widetilde{\cM}_n, \Q)$
is an equality.  The same holds for $\cM_n$.
\end{corollary}
\begin{proof}
We saw that the tautological classes lie in the Hodge subring, and Markman showed
that they generate the cohomology.  To deduce the statement for $\cM_n$, it suffices to 
note that $\mathrm{H}^*(\G_m, \Q) = \mathrm{Hdg}(\G_m)$. 
\end{proof}

\begin{corollary}
Assign bigradings $\deg(\alpha_k) = (k, 2k-2)$ and
$\deg(\phi_{k,i}) = (k, 2k-1)$ and $\deg(\beta_k) = (k, 2k)$
in order to view $\Q[\alpha_k, \phi_{k,i}, \beta_k] $ as a bigraded ring. 
On the other hand, view $\mathrm{H}^*(\widetilde{\cM}_n, \Q) = \bigoplus 
{}^k \mathrm{Hdg}^j (\widetilde{\cM}_n)$ as a bigraded ring with bigrading $(k,j)$.  

Then the surjective morphism
$\Q[\alpha_k, \phi_{k,i}, \beta_k] \to \mathrm{H}^*(\widetilde{\cM}_n, \Q)$ is 
a bigraded ring morphism.  In particular, its kernel -- the relations between the 
tautological generators -- is generated by bihomogenous elements. 
\end{corollary}

Hausel and Rodriguez-Villegas formulated a structural conjecture 
about the action of $\alpha_2$ \cite[Conj. 4.2.7]{HRV}.  They wrote 
in terms of the associated graded of the weight filtration, but
it can be formulated in other ways:

\begin{corollary} \label{cor:equiv}
Let $\alpha := \int_{\Sigma} c_2(\cT) \in {}^2 \mathrm{Hdg}^2$, and let 
$d = \dim\, \widetilde{\cM}_n$.  Then the following formulations of
the ``curious hard Lefschetz'' conjecture are equivalent: 
\begin{itemize}
\item After passing to associated graded with respect to the weight filtration:
$$\alpha^s \cup \cdot : Gr_{d-2s}^W \mathrm{H}^{*-s}(\cM_n, \C) 
 \xrightarrow{\sim}Gr_{d+2s}^W \mathrm{H}^{*+s}(\cM_n, \C)  $$
\item After passing to associated graded with respect to the Hodge filtration: 
$$\alpha^s \cup \cdot : Gr^{d/2-s}_F \mathrm{H}^{*-s}(\cM_n, \C) 
 \xrightarrow{\sim}Gr^{d/2+s}_F \mathrm{H}^{*+s}(\cM_n, \C)  $$
That is, the Hodge filtration $F$ is a shift of the Deligne-Jacobsen-Morozov
filtration of $\alpha \cup$. 
\item In terms of the bigrading:
$$\alpha^s \cup \cdot : {}^{d/2-s} \mathrm{Hdg}^{*-s} ( \widetilde{\cM}_n)
 \xrightarrow{\sim} {}^{d/2+s} \mathrm{Hdg}^{*+s} ( \widetilde{\cM}_n)$$
\end{itemize}
\end{corollary}

\begin{remark}
Corollary 1 and Corollary 3 can be obtained from results already in \cite{HRV}. 
Corollary 2 is new.
\end{remark}

This is called a ``curious'' Hard Lefschetz because it is taking place on an affine
variety, does not involve an ample class, and is symmetric in Hodge gradings rather than the 
homological grading.  
From the point of view of the character variety,
it is completely mysterious why such a symmetry should exist. 

Recently, de Cataldo, Hausel, and Migliorini formulated a very natural
conjecture to explain this mystery \cite{dCHM}.  There is  
a ``nonabelian Hodge'' diffeomorphism \cite{Cor, Don, Hit1, Sim} between the twisted character variety 
$\cM_n$ and Hitchin's moduli space of rank $n$ Higgs bundles $\cH_n$ on some 
complex curve with underlying surface $\Sigma$, giving in
particular a canonical identification $\mathrm{H}^*(\cM_n; \Q) = \mathrm{H}^*(\cH_n; \Q)$. 
The space $\cH_n$
admits the structure of a complex integrable system \cite{H}, and in particular carries a proper morphism
to an affine space $h: \cH_n \to \cA_n$.  This map gives rise to a perverse Leray filtration
 $P_k\mathrm{H}^*(\cH_n; \Q)$, and the ``P = W'' conjecture asserts that 
 $P_k\mathrm{H}^*(\cH_n; \Q) = W_{2k} \mathrm{H}^*(\cM_n; \Q)$.  There are likewise
 versions for the $SL(n)$ and the $PGL(n)$ cases.  The relation to the 
 curious Hard Lefschetz conjecture is that the relative Hard Lefschetz theorem \cite{BBD}
 guarantees
 $$\alpha^s \cup \cdot : Gr_{d/2-s}^P \mathrm{H}^{*-s}(\widetilde{\cH}_n, \C) 
 \xrightarrow{\sim}Gr_{d/2+s}^P \mathrm{H}^{*+s}(\widetilde{\cH}_n, \C)$$ 
I note in passing that the equivalence of the 
various conditions of Corollary \ref{cor:equiv} would then imply that the various possible splittings 
of the perverse filtration all agree \cite{dC}. 

The $A_1$ case of the ``P=W'' conjecture was established in \cite[Thm. 1.1.1]{dCHM}.  In broad outline,
their proof proceeds as follows. First, they showed
that the tautological classes $\alpha_k, \beta_k, \phi_{k,i} \in \mathrm{H}^*(\cH_n, \Q)$ 
occupy the same place in the perverse filtration that they do in the weight filtration.  Second, 
they showed the cup product respects the perverse filtration.  To generalize this approach
to higher rank, the first step is to compute the weights and perverse degrees of the tautological
classes.  In this article we have computed the weights.

\subsection*{Appendix: review of simplicial terminology}

The above constitutes a complete argument, but for those still seeking homotopical enlightenment, 
we recall the meanings of some of the words.  Everything which follows is standard material, for which there 
are many references; we found \cite{C} especially helpful.

We write $\Delta$ for the simplex category -- its objects are the nonempty finite totally ordered sets, and its morphisms 
are order preserving maps.  Evidently the objects are each uniquely isomorphic to some 
$[n] = [0 \to 1 \to \cdots \to n]$.  The maps are generated by the $n + 1$ ``include a face'' maps $[n-1] \to [n]$ and the
$n$ ``degenerate an edge'' maps $[n] \to [n-1]$, subject to some relations.   That is, the category looks like this: 

$$ [0]   \mathrel{\vcenter{\offinterlineskip%
  \hbox{$\rightarrow$}\vspace{.1ex}\hbox{$\leftarrow$}\vspace{.1ex}\hbox{$\rightarrow$}}}
[1]  \mathrel{\vcenter{\offinterlineskip%
  \hbox{$\rightarrow$}\vspace{.1ex}\hbox{$\leftarrow$}\vspace{.1ex}\hbox{$\rightarrow$}\vspace{.1ex}\hbox{$\leftarrow$}\vspace{.1ex}\hbox{$\rightarrow$}  }} [2] \cdots$$

A ``simplicial object'' in a category $\cC$ is just 
a functor $\Delta^{op} \to \cC$.  That is, it looks like this:

$$ 
K_0   \mathrel{\vcenter{\offinterlineskip%
  \hbox{$\leftarrow$}\vspace{.1ex}\hbox{$\rightarrow$}\vspace{.1ex}\hbox{$\leftarrow$}}}
K_1  \mathrel{\vcenter{\offinterlineskip%
  \hbox{$\leftarrow$}\vspace{.1ex}\hbox{$\rightarrow$}\vspace{.1ex}\hbox{$\leftarrow$}\vspace{.1ex}\hbox{$\rightarrow$}\vspace{.1ex}\hbox{$\leftarrow$} }} K_2 \cdots$$

The example which most concerns us here is the simplicial set (or space, or scheme) $BG$.  

$$ BG: 
pt  \mathrel{\vcenter{\offinterlineskip%
  \hbox{$\leftarrow$}\vspace{.1ex}\hbox{$\rightarrow$}\vspace{.1ex}\hbox{$\leftarrow$}}}
G  \mathrel{\vcenter{\offinterlineskip%
  \hbox{$\leftarrow$}\vspace{.1ex}\hbox{$\rightarrow$}\vspace{.1ex}\hbox{$\leftarrow$}\vspace{.1ex}\hbox{$\rightarrow$}\vspace{.1ex}\hbox{$\leftarrow$} }} G \times G \cdots$$

That is, $BG_n = G^n$, the left-going maps are given by inclusion into a factor, and the right-going maps are projection or multiplication.  For instance, the two maps $G \to G \times G$ are
$g \mapsto (g, 1)$ and $g \mapsto (1,g)$, and the three maps $G \times G \to G$ are $(g, h) \mapsto g$, $(g, h) \mapsto gh$, and $(g, h) \mapsto h$. 

A Cech cover $\{U_\alpha\}$ of a space $X$ determines a simplicial set $\mathfrak{U}_X$ by taking 
$$\mathfrak{U}_{X,i} = \{(\alpha_1, \alpha_2, \ldots, \alpha_n)\, | \, \bigcap U_{\alpha_i} \ne \emptyset \}$$ 
The face and degeneracy maps are given by omitting and doubling indices.

A map of simplicial sets $E: \mathfrak{U}_X \to BG$, in degree one, corresponds to a specification of an element of $G$ for every double
overlap $U_\alpha \cap U_\beta$.  When $\alpha = \beta$, this element must be the identity, for compatibility with the degeneracy map
from degree zero.  In degree two, we should give an element $(g, h) \in G \times G$ for each triple overlap $U_\alpha \cap U_\beta \cap U_\gamma$.  
The face maps assert that, in this case, $g$ should be the element assigned to $U_\alpha \cap U_\beta$, that $h$ should be the element assigned
to $U_\beta \cap U_\gamma$, and $gh$ should be the element assigned to $U_\alpha \cap U_\gamma$.  It turns out that associativity of $G$ then
determines all higher morphisms.  In other words, such a map determines a locally trivial $G$-bundle on $X$, trivialized on each of the $U_\alpha$. 
If each $U_\alpha$ is small enough that all locally trivial $G$-bundles are trivial, this gives a bijection: 
$$\Hom_{ssets}( \mathfrak{U}_X ,  BG ) \leftrightarrow \{\mbox{locally trivial $G$-bundles on $X$, trivialized on each of the $U_\alpha$}\}$$

The stack of local systems parameterizes bundles without the data of trivializations on the charts.  
In the simplicial setting, this corresponds to promoting the 
LHS from the ordinary Hom, which is a set, to the internal Hom, which is a simplicial set.  
To describe this, let $\Delta_n$ be the ``simplicial $n$-simplex'', i.e., the functor on $\Delta$ given by 
$$\Delta_n ( \,\cdot\, ) = \Hom_{\Delta}(\,\cdot \, , [n])$$

By Yoneda, this object has the feature that $\Hom_{ssets}(\Delta_n, X) = X([n]) = X_n$.  This leads to the definition of a simplicial
set $\underline{\Hom}_{ssets}(X, Y)$ whose $n$-simplices are 
$$\underline{\Hom}_{ssets}(X, Y)_n = \Hom_{ssets}(\Delta_n, \underline{\Hom}_{ssets}(X, Y)) = 
\Hom_{ssets}(\Delta_n \times X, Y)$$

The product of simplicial sets is the usual product of functors, i.e., $(X \times Y)_n = X_n \times Y_n$, and similarly on maps. Chasing a 
(fairly large) diagram shows that 
$\Hom_{ssets}(\mathfrak{U}_X ,  BG)_1$ is the set of triples $(\cE, \cF, \phi)$ where $\cE, \cF \in \Hom_{ssets}(\mathfrak{U}_X ,  BG)_1$, and $\phi$ 
is a map $\mathfrak{U}_{X,1} \to G$ so that changing the trivialization accordingly carries $\cE$ to $\cF$.  
The two maps to $\Hom_{ssets}(\mathfrak{U}_X ,  BG)_0$ are just the restrictions to $\cE$ and $\cF$.  That is, 
$$ Loc_G(X): \Hom_{ssets}( \mathfrak{U}_X ,  BG )_0 \leftleftarrows \Hom_{ssets}( \mathfrak{U}_X ,  BG )_1$$ 
is a usual presentation of the groupoid of local systems.  

If we view $\mathfrak{U}_X$ as a simplicial scheme by
just viewing each element of $\mathfrak{U}_{X,i}$ as a copy of $\mathrm{Spec}\, \C$, 
we find the usual presentation of the (1-)stack of local systems:
$$Loc_G(X):  \Hom_{sschemes}( \mathfrak{U}_X ,  BG )_0 \leftleftarrows \Hom_{sschemes}( \mathfrak{U}_X ,  BG )_1$$ 

In fact, $\underline{\Hom}_{sschemes}(\Delta_X, BG)$ has vanishing higher homotopy groups, 
so mapping to the 1-truncation gives $$\underline{\Hom}_{sschemes}(\Delta_X, BG) \xrightarrow{\sim} Loc_G(X)$$  
From a modern point of view this is 
rather the {\em definition} of $Loc_G(X)$, and then a calculation of homotopy groups reveals that it is a 1-stack rather
than a higher stack.  For a derived version of this (not needed here, but which specializes to the underived calculation), 
see \cite[Lemma 2.2.6.3]{HAG2}. 

\subsection*{Acknowledgements}  I would like to thank Mark de Cataldo, Tam\'as Hausel, and Luca Migliorini for helpful and encouraging 
conversations around these topics, Ben Davison for discussions about simplicial technology and finding 
\cite[Lemma 2.2.6.3]{HAG2}, and Geordie Williamson for correcting an egregiously false belief I had about mixed Hodge theory. 
I am partially supported by the NSF grant DMS-1406871.

\end{document}